\documentclass[12pt, twoside, leqno]{article}

\usepackage{amsmath,amsthm}
\usepackage{amssymb}

\usepackage{enumitem}

\usepackage{graphicx}

\usepackage[T1]{fontenc}

\newtheorem{theorem}{Theorem}[section]
\newtheorem{corollary}[theorem]{Corollary}

\newtheorem{proposition}[theorem]{Proposition}

\theoremstyle{definition}
\newtheorem{definition}[theorem]{Definition}

\newtheorem{example}[theorem]{Example}

\numberwithin{equation}{section}

\begin{document}


\baselineskip=17pt


\title{Cantorvals as sets of subsums for a series connected with trigonometric functions}

\author{Mykola Pratsiovytyi\\
Institute of Mathematics\\
National Academy of Sciences of Ukraine\\
Tereschenkivska 3\\
01024 Kyiv, Ukraine\\
E-mail: pratsiovytyi@imath.kiev.ua
\and
Dmytro Karvatskyi\\
Institute of Mathematics\\
National Academy of Sciences of Ukraine\\
Tereschenkivska 3\\
01024 Kyiv, Ukraine\\
E-mail: karvatsky@imath.kiev.ua}

\date{}

\maketitle


\renewcommand{\thefootnote}{}

\footnote{2020 \emph{Mathematics Subject Classification}: 40A05, 11B05, 28A75.}

\footnote{\emph{Key words and phrases}: achievement set of sequence, multigeometric series, the set of subsums, Cantorval.}

\renewcommand{\thefootnote}{\arabic{footnote}}
\setcounter{footnote}{0}


\begin{abstract}
We study properties of the set of subsums for convergent series $~~~~~ k_1 \sin x + \dots + k_m \sin x +\dots + k_1\sin x^n +\dots + k_m \sin x^n + \dots $, \newline
where $k_1, k_2, k_3,\dots,k_m$ are fixed positive integers and $0<x<1$. ~ ~ ~ ~ Depends on parameter $x$ this set can be a finite union of closed intervals or Cantor-type set or even Cantorval.
\end{abstract}

\section{Introduction}

\begin{definition}
Let $M\subseteq N$. Then number
\begin{equation}
 \label{incomplit sum}
x=x\left(M\right)=\sum_{n\in M\subseteq N}{u_n}=\sum^{\infty }_{n=1}{{\varepsilon }_nu_n}, ~ ~ ~ {\varepsilon }_n=\left\{ \begin{array}{c}
1 \ \text{ for }  \ n\in M,\  \\
0 \ \text{ for }  \ n\notin M,\  \end{array}
\right.
\end{equation}
is called the subsum of the series $\displaystyle \sum{u_n}$. By $E(u_n)$ we denote the set of all subsums \eqref{incomplit sum} for this series. This mathematical object is also defined by some scientists as achievement set of the sequence $(u_n)$.
\end{definition}
Investigation of subsums for a convergent positive, absolutely convergent, divergent, complex series has a very deep history. In the paper \cite{Kakeya} Japanese mathematician Soichi Kakeya proved that the set $E(u_n)$ for a positive convergent series is:

\begin{itemize}
\item a perfect set;
\item a finite union of closed intervals if relation $\displaystyle u_{n} \leq r_{n}=\sum_{i=n+1}^{\infty}u_i$
holds for sufficiently large $n$;
\item homeomorphic to the classic Cantor set (Cantor-type set) if relation \\ $\displaystyle u_{n} > r_{n}=\sum_{i=n+1}^{\infty}u_i$
holds for sufficiently large $n$.
\end{itemize}

In the same paper Kakeya formulated a hypothesis that for any positive convergent series $E(u_n)$ is rather a Cantor-type set or a finite union of closed intervals. However, this conjecture was false. The simplest example was presented by American mathematicians J. Guthrie and J. Nymann \cite{Guthrie 1988}:
The set $T$ of subsums for Guthrie-Nymann's series
\begin{equation}
\label{GNS}
\frac{3}{4}+\frac{2}{4}+\frac{3}{4^2}+\frac{2}{4^2}+\frac{3}{4^3}+\frac{2}{4^3}+\dots+\frac{3}{4^n}+\frac{2}{4^n}+\dots
\end{equation}
contains the interval $\left[ 3/4, 1 \right]$, but it is not a finite union of closed intervals.
This set $T$ has properties of nowhere dense set and infinite union of closed intervals and it is called Cantorval.
Every Cantorval is homeomorphic to the set $T^*$ which is defined as follows

$$T^* \equiv C \cup \bigcup_{n=1}^{\infty} G_{2n-1}=[0, 1] \setminus \bigcup_{n=1}^{\infty} G_{2n},$$
where $C$ is the Cantor ternary set, $G_k$ is the union of the $2^{n-1}$ open middle thirds which are removed from $[0, 1]$ at the $n$-th step in the construction of $C$.
As a result of articles \cite{Guthrie 1988} and \cite{Nymann}, we have the following theorem.

\begin{theorem}
\label{GN-Theorem}
The set of subsums for a convergent positive series is one of the following three types: 
\begin{enumerate}
\item a finite union of closed intervals;
\item homeomorphic to the Cantor set;
\item M-Cantorval (homeomorphic to the set $T$).
\end{enumerate}
\end{theorem}

From the other hand, Cantorvals naturally can be defined as the product
$$C_1 + C_2=\{ x_1 + x_2; x_1 \in C_1, x_2 \in C_2\}$$
of the arithmetic sum of two Cantor sets \cite{Mendes}.

\section{Multigeometric sequences and related series}

The necessary and sufficient condition that the set of subsums is a Cantorval or homeomorphic to the Cantor set is still unknown. Despite essential progress for some series, the problem is quite difficult in general case. In this context scientists focus on series such that their terms are elements of some sequence \cite{PKA} with some condition of homogeneity (depend on finite numbers of parameters and defined by a formula for general term or  recurrence relation).

A large number of articles are devoted to the study of topological and metric properties of the set of subsums for the series
\begin{equation}
\label{MGS}
\sum_{n=1}^{\infty} v_n = k_1+ \dots +k_m+k_1q + \dots +k_mq+\dots +k_1q^{n-1} +\dots +k_mq^{n-1} + \dots,
\end{equation}
where $k_1, k_2, \dots k_m$ are fixed positive integers, $q \in \left(0, 1\right)$. Terms of this series are an elements of multigeometric sequence.
In the paper \cite{Bartoszewicz} some conditions for $E(v_n)$ to be a Cantorval or a Cantor-type set were found.

\begin{theorem}{(\cite{Bartoszewicz})}
Let $k_1 \geq k_2 \dots \geq k_m$ and $K=\sum_{i=1}^{m}k_i$. Assume that there exist positive integers $n_0$ and $n^*$ such that each of the numbers $n_0, n_0+1,$ $n_0+2, \dots, n_0+n^*$ can be obtained by summing up the numbers $k_1, k_2, \dots, k_m$. Then the set $E(v_n)$ satisfies following:
\begin{itemize}
\item if $q \geq1/(n^*+1)$ then $E(v_n)$ contains an interval;
\item if $q <k_m/(K+k_m)$ then $E(v_n)$ is not a finite union of closed interval;
\item if $1/(n^*+1) \leq q <k_m/(K+k_m)$ then $E(v_n)$ is a Cantorval.

\end{itemize}

\end{theorem}

\begin{theorem}{(\cite{Bartoszewicz})}
Let the set $\Sigma$ is defined as
$$\Sigma=\left\{ \sum_{i=1}^{m} \varepsilon_i k_i : (\varepsilon_i)_{i=1}^{m} \in \{ 0, 1 \}^{m} \right\}.$$
If $q<1/card \Sigma$ then $E(v_n)$ is a Cantor-type set.
\end{theorem}

This result was generalized for case $k_1=a, k_2=a+d, k_3=a+2d, \dots,$ $k_m=a+cd$ for some integers parameters $a, b$ and $c$ in \cite{Ferdinands}.
In the same paper authors described another interesting family of Cantorvals generated by multigeometric sequence.

\begin{theorem}{(\cite{Ferdinands})}
Let we have $k_1=a+2nd, k_2=a+(2n-2)d, \dots,$ $k_{m-2}=a+2d, k_{m-1}=a, k_m=d$ with $2nd < a < (2n+2)d$ and $n \geq 4$.
If
$$\frac{1}{2n+2} \leq q < \min\left\{\frac{d}{a}, \frac{a-d}{(n+2)a+(n^2+n)d} \right\}$$
then $E(v_n)$ is a Cantorval.
\end{theorem}

Investigation of sets with difficult local structure, such as cantorvals or fractals, is a trend in modern mathematical research.
In particular, for Guthrie-Nymann's Cantorval and its modifications were solved some topological, metric and probabilistic problems \cite{Bielas}, \cite{PKB}. The paper \cite{Banakh} is devoted to study of some self-similar sets having the same complicated properties as $E(v_n)$.

\section{Trigonometric function and connected series}

Let us consider the series
\begin{equation}
\label{SIN}
\sum_{n=1}^{\infty} a_n = k_1 \sin{x} + \dots + k_m \sin{x} + \dots + k_1 \sin{x^n} + \dots + k_m \sin{x^n} + \dots,
\end{equation}
where $k_1 \geq k_2 \geq \dots \geq k_m$ are fixed positive integers and $x \in \left(0, 1 \right)$.

We define $K=k_1+k_2+\dots+k_m$. For $0<x<1$ the series \eqref{SIN} is convergent as from inequality $\sin{x}<x$ we have that
$$K\sum_{n=1}^{\infty} \sin{x^n} < K\sum_{n=1}^{\infty} x^n=\frac{Kx}{1-x}.$$

We study properties of the set $E(a_n)$ for the series \eqref{SIN} depend on initial parameters $k_1, \dots, k_m, x$.

\begin{theorem}
\label{I}
If the series \eqref{SIN} satisfies the condition
$$x \geq \frac{d}{K + d},$$
$$\text{where} ~ d=\max_{1 \leq j \leq m}{d_j}, ~ d_j=\frac{\pi}{2}k_j-k_{j+1}-k_{j+2}-\dots-k_m,$$
then $E(a_n)$ is an interval.
\end{theorem}
\begin{proof}
For any $n=im+j$, where $i \in N_0$, $j=\overline{1, m}$ we have
$$a_{im+j}=k_j \cdot \sin{x^{i+1}},$$
$$r_{im+j}=k_{j+1}\sin{x^{i+1}}+k_{j+2}\sin{x^{i+1}}+\dots +k_m\sin{x^{i+1}}+K\sum_{p=i+2}^{\infty} \sin{x^{p}}.$$
Taking into account Jordan's inequality \cite{Jordan's}
$$\frac{2x}{\pi} < \sin x < x ~ ~ ~ \text{for} ~ ~ ~ 0 < x < \frac{\pi}{2}, $$
we get
$$a_{im+j}<k_j x^{i+1},$$
$$r_{im+j}>\frac{2}{\pi} \left(K \frac{x^{i+2}}{1-x}+k_{j+1} x^{i+1} + \dots + k_m x^{i+1} \right).$$

Thus, for $x \geq d_j/(K+d_j)$, where $d_j=\frac{\pi}{2}k_j-k_{j+1}-k_{j+2}-\dots-k_m$ the following inequality holds
$$a_{im+j}<k_j x^{i+1}< \frac{2}{\pi} \left(K \frac{x^{i+2}}{1-x}+k_{j+1} x^{i+1} + \dots + k_m x^{i+1} \right)<r_{im+j}.$$

Moreover, if $x \geq d/(K+d)$, where $\displaystyle d=\max_{1 \leq j \leq m}{d_j}$, then inequality $a_{n}<r_{n}$ holds for any $n \in N$. Then due to Kakeya's result the set $E(a_n)$ is an interval.

\end{proof}

\begin{example}

For $k_1=2, k_2=1$ and $x=1/2$ the series \eqref{SIN} has form
$$2\sin\left(\frac{1}{2}\right)+\sin\left(\frac{1}{2}\right)+2\sin\left(\frac{1}{2}\right)^2+ \dots +2\sin\left(\frac{1}{2}\right)^n+\sin\left(\frac{1}{2}\right)^n+\dots.$$
In this case $K=3$ and $d_1=\pi-1, d_2=\pi/2$. Then, according to the Theorem \ref{I}, for any values of parameter $x$ satisfying the inequality
$$x \geq \frac{\left(\pi-1\right)}{3+\left(\pi-1\right)},$$
the set $E(a_n)$ is an interval.

\end{example}

\begin{theorem}
If the series \eqref{SIN} satisfies the condition
$$x \leq \frac{2k_m}{K\pi + 2k_m},$$
then $E(a_n)$ is not a finite union of closed intervals.
\end{theorem}

\begin{proof}

To prove the Theorem we will show that inequality $a_n>r_n$ holds for infinite numbers of $n$. Taking into account Jordan's inequality, for any $i \in N$ we have
$$r_{mi}=K\sum_{p=i+1}^{\infty} \sin{x^p}<K\sum_{p=i+1}^{\infty} {x^p}=\frac{K x^{i+1}}{1-x},$$
$$a_{mi}=k_m\sin{x^i}>\frac{2k_m x^i}{\pi}.$$
It is easy to see that for $x \leq 2k_m / (K\pi + 2k_m)$ the following inequality holds
$$r_{mi}<\frac{K x^{i+1}}{1-x}<\frac{2k_m x^i}{\pi}<a_{mi}.$$
So $a_{mi}>r_{mi}$ for any $i \in N$, hence $E(a_n)$ is not a finite union of intervals.
\end{proof}

\begin{theorem}
If for the series \eqref{SIN} exist such numbers $n_0$ and $n^*$ such, that every number $n_0, n_0+1, n_0+2, \dots n_0+n^*$ can be obtained by summing up the numbers $k_1, k_2, \dots, k_m$ and parameter $x$ satisfies the inequality
$$x \geq \frac{\pi}{2n^*+\pi},$$
then $E(a_n)$ contains intervals.
\end{theorem}

\begin{proof}
According to the statement for any $h \in \{0, 1, 2, \dots n^* \}$ there exists sequence $(c_i)^{m}_{i=1}$ of 0 and 1 such that
$$n_0+h=\sum_{i=1}^{m}c_i k_i.$$
From it follows that every number $t$ having the representation
$$t=\sum_{i=1}^{\infty} (n_0+h_i) \sin{x^i}=n_0 \sum_{i=1}^{\infty} \sin{x^i} + \sum_{i=1}^{\infty} h_i \sin{x^i}, ~ ~ ~ h_i \in \{0, 1, 2, \dots, n^*\},$$
is contained on $E(a_n)$. The set of all such numbers $t$ is an arithmetic sum of a constant and the set of subsums for the series
$$\sum_{n=1}^{\infty}{b_n}=\underbrace{\sin{x}+\dots+\sin{x}}_{n^*}+ \dots + \underbrace{\sin{x^{\left[\frac{n-1}{n^*}+1\right]}}+\dots + \sin{x^{\left[\frac{n-1}{n^*}+1\right]}}}_{n^*}+\dots,$$
where every term has $n^*$ repetitions. It is obvious, that $\displaystyle \sum_{p=n+1}^{\infty}b_p=r^*_n>b_n$ for $n \neq in^*$, $i \in N$. Taking into account Jordan's inequality, we have
$$b_{n^*i}=\sin{x^i}<x^i,$$
$$r^*_{n^*i}=n^*(\sin{x^{i+1}}+\sin{x^{i+2}}+ \dots)>n^*\left(\frac{2x}{\pi}+\frac{2x^2}{\pi}+\dots\right)=\frac{2n^*x}{\pi(1-x)}.$$
It is easy to show that for $x \geq \pi/(2n^*+\pi)$ the following inequality holds

$$b_{mi}<x^i<\frac{2n^*x^{i+1}}{\pi(1-x)}<r^*_{mi}.$$
Thus, $b_{n}<r^*_{n}$ for any $n \in N$. Then, due to Kakeya's result the set $E(b_n)$ is a interval. Since
$$n_0 \sum_{i=1}^{\infty} \sin{x^i} + E(b_n) \subset E(a_n),$$
then the set $E(a_n)$ contains intervals.
\end{proof}

\begin{corollary}
\label{MC}
If for the series \eqref{SIN} exist the numbers $n_0$ and $n^*$ such that every number $n_0, n_0+1, n_0+2, \dots, n_0+n^*$ can be obtained by summing up the numbers $k_1, k_2, \dots, k_m$ and parameter $x$ satisfies the inequality
$$\frac{\pi}{2n^*+\pi} \leq x \leq \frac{2k_m}{K\pi + 2k_m},$$
then the set $E(a_n)$ is a Cantorval.
\end{corollary}

\begin{example}
For $k_1=8, k_2=7, k_3=6, k_4=5, k_5=4$ and $x=1/15$ the series \eqref{SIN} has form

$$8\sin{\frac{1}{15}}+7\sin{\frac{1}{15}}+6\sin{\frac{1}{15}}+5\sin{\frac{1}{15}}+4\sin{\frac{1}{15}}+\dots$$
$$\dots+8\sin{\frac{1}{15^n}}+7\sin{\frac{1}{15^n}}+6\sin{\frac{1}{15^n}}+5\sin{\frac{1}{15^n}}+4\sin{\frac{1}{15^n}}+\dots$$
In this case $K=30$ and $n_0=4, n^*=22$. Then, according to the Corollary \ref{MC}, for any values of parameter $x$ satisfying the condition
$$\frac{\pi}{44+\pi} \leq x \leq \frac{8}{30\pi+8}$$
the set of subsums $E(a_n)$ is a Cantorval.
\end{example}

\begin{theorem}
\label{TC}
If the series \eqref{SIN} satisfies the condition
$$x \leq \frac{l}{K + l},$$
$$\text{where} ~ l=\min_{1 \leq j \leq m}{l_j}, ~ l_j=\frac{2}{\pi}k_j-k_{j+1}-k_{j+2}-\dots-k_m,$$
then $E(a_n)$ is a Cantor-type set.
\end{theorem}
\begin{proof}
For any $n=im+j$, where $i \in N_0$, $j=\overline{1, m}$ the following inequality holds
$$a_{im+j}=k_j \cdot \sin{x^{i+1}},$$
$$r_{im+j}=k_{j+1}\sin{x^{i+1}}+k_{j+2}\sin{x^{i+1}}+\dots +k_m\sin{x^{i+1}}+K\sum_{p=i+2}^{\infty} \sin{x^{p}}.$$
Taking into account Jordan's inequality, we have
$$a_{im+j}>\frac{2}{\pi}k_j x^{i+1},$$
$$r_{im+j}<K \frac{x^{i+2}}{1-x}+k_{j+1} x^{i+1} + \dots + k_m x^{i+1}.$$

Thus, for $x \leq l_j/(K+l_j)$, where $l_j=\frac{2}{\pi}k_j-k_{j+1}-k_{j+2}-\dots-k_m$ the following inequality holds
$$a_{im+j}>\frac{2}{\pi}k_j x^{i+1} > K \frac{x^{i+2}}{1-x}+k_{j+1} x^{i+1} + \dots + k_m x^{i+1} > r_{im+j}.$$

Moreover, if $x \leq l/(K+l)$, where $\displaystyle l=\min_{1 \leq j \leq m}{l_j}$, then $a_{n}>r_{n}$ for any $n \in N$. Then, due to Kakeya's result the set  $E(a_n)$ is a Cantor-type set.
\end{proof}

\begin{example}

For $k_1=4, k_2=1$ and $x=1/10$ the series \eqref{SIN} has form

$$4\sin\left(\frac{1}{10}\right)+\sin\left(\frac{1}{10}\right)+\dots +4\sin\left(\frac{1}{10}\right)^n+\sin\left(\frac{1}{10}\right)^n+\dots.$$
In this case $K=5$ and $l_1=8/\pi-1, l_2=2/\pi$. Then, according to the Theorem \ref{TC}, for any values of parameter $x$ satisfying the condition
$$x \leq \frac{2/\pi}{5+2/\pi}$$
the set $E(a_n)$ is a Cantor-type set.

\end{example}

\section{Open problems and hypothesis}

As a result of the paper we find some condition for $E(a_n)$ to be an interval, Cantor-type set or Cantorval. Hoverer, there still exist a few open problems:
\begin{itemize}
\item for some values of parameter $x$ topological type of $E(a_n)$ is unknown (for $2k_m/(\pi K+2k_m)< x <d/(K+d)$ as an example);

\item calculation of Lebesgue measure for Cantorvals generated by main series;

\item we can't still find the necessary and sufficient condition for $E(a_n)$ to have a zero Lebesgue measure;

\item calculation of the Hausdorff-Besicovitch dimension of $E(a_n)$ if it has zero Lebesgue measure.

\end{itemize}

We described some properties of the set of subsums for the series \eqref{SIN} in the Section 3 by using properties of its terms $\displaystyle k_j \frac{2x^n}{\pi} < k_j \sin x^n < k_j x^n$ for any $n \in N$, $j \in \{1, 2,\dots, m \}$ and $0<x<1$. Let's remark that terms of the series \eqref{SIN} are located between terms of two multigeometric series \eqref{MGS} with $q_1=2x/\pi$ and $q_2=x$.
Taking into account such considering, we can formulate main hypothesis.
\begin{proposition}
Let $\sum_{n=1}^{\infty} a_n$ and $\sum_{n=1}^{\infty} b_n$ be a convergent positive series with the same types of the set of subsums (a finite union of intervals, a Cantor-type set or a M-Cantorval).
If $\sum_{n=1}^{\infty} c_n$ satisfies condition
\begin{equation}
\left[
\begin{array}{l}
\displaystyle a_n \leq c_n \leq b_n \leq r^{a}_n=\sum_{i=n+1}^{\infty}{a_i} \leq r^{c}_n=\sum_{i=n+1}^{\infty}{c_i} \leq r^{b}_n=\sum_{i=n+1}^{\infty}{b_i}, \\
\displaystyle r^{a}_n=\sum_{i=n+1}^{\infty}{a_i} < r^{c}_n=\sum_{i=n+1}^{\infty}{c_i} < r^{b}_n=\sum_{i=n+1}^{\infty}{b_i} < a_n < c_n < b_n, \\
\end{array}
\right.
\end{equation}
for any $n \in N$ then $E(a_n)$, $E(b_n)$ and $E(c_n)$ have the same topological type.
\end{proposition}

It is easy to prove this statement for case when $E(a_n)$ and $E(b_n)$ are both a finite union of closed intervals. We also can prove the hypothesis for case when inequalities $a_n \leq r^{a}_{n}$ and $b_n \leq r^{b}_{n}$ hold only for finite numbers of $n$. However, in general case this hypothesis is not obvious and not so easy to prove.

\subsection*{Acknowledgements}

The authors would like to thank Professor Aniceto Murillo and the administration of the University of Malaga for the invitation and opportunity to continue our mathematical research in Spain during such difficult period for Ukrainian science.

\end{document}